\documentclass[a4paper]{amsart}
\usepackage{amssymb,amsfonts}
\usepackage[all,arc]{xy}
\usepackage{enumerate}
\usepackage{mathrsfs}
\usepackage{mathtools}
\usepackage{tikz}
\usepackage{pgfplots}
\usetikzlibrary{calc}
\usetikzlibrary{arrows,snakes,backgrounds}
\usepackage{color}   
\usepackage{marginnote}
\usepackage{tikz-cd}
\usetikzlibrary{positioning}

\usepackage{enumitem}
\usepackage{hyperref}
\usepackage{MnSymbol} 
\hypersetup{
    colorlinks=true, 
    linktoc=all,     
    citecolor=black,
    filecolor=black,
    linkcolor=blue,
    urlcolor=black
}
\usepackage{cleveref}
\usetikzlibrary{arrows}

\usepackage[normalem]{ulem}

\newtheorem{thm}{Theorem}[section]
\newtheorem{cor}[thm]{Corollary}
\newtheorem{prop}[thm]{Proposition}
\newtheorem{lem}[thm]{Lemma}

\newtheorem{quest}[thm]{Question}

\newtheorem{problem}[thm]{Problem}
\newtheorem*{openproblem*}{Problem}
\newtheorem*{quest*}{Question}
\newtheorem*{problem*}{Problem}

\theoremstyle{definition}
\newtheorem{defn}[thm]{Definition}

\theoremstyle{remark}
\newtheorem{rem}[thm]{Remark}

\newcommand{\bC}{\mathbb{C}}

\newcommand{\bQ}{\mathbb{Q}}
\newcommand{\bR}{\mathbb{R}}

\newcommand{\bZ}{\mathbb{Z}}

\newcommand\Diff{\mathrm{Diff}}
\newcommand\SDiff{\mathrm{SDiff}}
\newcommand\Homeo{\mathrm{Homeo}}
\newcommand\BDiff{\mathrm{BDiff}}
\newcommand\BHomeo{\mathrm{BHomeo}}

\newcommand{\hcoker}{/\!\!/}

\newcommand{\SO}{\mathrm{SO}}
\newcommand{\fix}{\mathrm{Fix}}
\newcommand{\GL}{\mathrm{GL}}

\renewcommand{\paragraph}[1]{\medskip \noindent {\bf #1.}}

\newcounter{notes}

\addtocontents{toc}{\protect\setcounter{tocdepth}{1}}
\makeatletter
\let\c@equation\c@thm
\makeatother
\numberwithin{equation}{section}
\title[Dynamical and cohomological obstructions]{Dynamical and cohomological obstructions to extending group actions}

\author{Kathryn Mann}
\address{Department of Mathematics\\
Cornell University\\
Ithaca, NY, 14850 }
\email{k.mann@cornell.edu}
\author{Sam Nariman}
\address{Department of Mathematics\\
  University of Copenhagen\\
Universitetsparken 5\\
2100 Copenhagen, Denmark
}
\email{sam@math.ku.dk}

\begin{document}

\begin{abstract}
Motivated by a question of Ghys, we study obstructions to extending group actions on the boundary $\partial M$ of a $3$-manifold to a $C^0$-action on $M$.   
Among other results, we show that for a $3$-manifold $M$, the $S^1 \times S^1$ action on the boundary does not extend to a $C^0$-action of $S^1 \times S^1$ as a discrete group on $M$, except in the trivial case $M \cong D^2 \times S^1$.  
\end{abstract}
\maketitle

\section{Introduction}
This paper concerns the structure of diffeomorphism and homeomorphism groups of manifolds.  Our motivation is the following seemingly simple question of Ghys.  

\begin{quest}[\cite{Ghys}]
\textit{If M is a manifold with boundary $\partial M$, under what conditions is there a homomorphism $\Diff_0(\partial M) \to \Diff_0(M)$ that ``extends  $C^{\infty}$-diffeomorphisms to the interior"?}
\end{quest} 
Here and in what follows, $\Diff(M)$ denotes the group of self-diffeomorphisms of $M$, and $\Diff_0(M)$ its identity component.  
Put otherwise, Ghys' question asks for obstructions to a group-theoretic section of the natural ``restriction to boundary'' map $r\colon \Diff_0(M) \to \Diff_0(\partial M)$.   Restricting the domain of the map to $\Diff_0(M)$ ensures that the boundary map is surjective onto $\Diff_0(\partial M)$, thus any such obstruction will necessarily be group-theoretic in nature.  
Ghys' original work treats the case where $M$ is an $n$-dimensional ball (for general $n$), a case explored further in \cite{Mann15}.  This program is also reminiscent of Browder's notion of {\em bordism} of diffeomorphisms, as discussed at the end of this section.

We are interested in a more general problem of obstructions for {\em extending group actions}. Given a discrete group $\Gamma$ and a homomorphism $\rho\colon \Gamma \to \Diff_0(\partial M)$, an \emph{extension of $\rho$ to $M$} is a homomorphism $\phi$ such that the following diagram commutes.
\begin{displaymath}
    \xymatrix @M=3pt {
          & \Diff_0(M) \ar[d]^{r} \\
        \Gamma \ar[r]_{\hspace{-.7cm}\rho} \ar[ur]^\phi  & \Diff_0(\partial M) }
\end{displaymath}
This is already a challenging and interesting problem when $M$ is a 3-manifold with sphere or torus boundary.  Our focus here is cohomological obstructions to extension: we interpret Ghys' question as an {\em invitation to understand the relationship between the group cohomology of $\Diff_0(M)$ and $\Diff_0(\partial M)$. }
 
When $\partial M$ is diffeomorphic to the torus, the cohomology of the classifying space of $\Diff_0(\partial M)$ with its $C^{\infty}$-topology is known (see \cite{earle1969fibre}) to be  $$H^*(\BDiff_0(\partial M);\bQ)\cong H^*(\BHomeo_0(\partial M);\bQ)\cong \bQ[x_1,x_2].$$ Therefore, there are two potential obstruction classes $\rho^*(x_1), \rho^*(x_2)\in H^2(\mathrm{B}\Gamma;\bQ)$ where $\mathrm{B}\Gamma$ is the classifying space of the group $\Gamma$.   Like Ghys, we assume the group homomorphisms are between groups equipped with the discrete topology.   

We prove the following theorem on the restriction map, which allows us to find obstruction classes for the section problem over various subgroups.  
\begin{thm}\label{main}
Let $M$ be an orientable three-manifold which is not diffeomorphic to $D^2 \times S^1$, and with $\partial M$ diffeomorphic to  $T^2$. The map
\[H^2(\BDiff_0(T);\bQ)\to H^2(\BDiff_0(M);\bQ),\]
which is induced by the restriction map $\Diff_0(M)\to \Diff_0(T)$, has a nontrivial kernel.
The same holds when $\Diff_0$ is replaced by $\Homeo_0$.
\end{thm}

If $M$ is irreducible (and therefore Haken) this is a consequence of Waldhausen's theorem (\cite{MR0236930}) on Seifert fibered manifolds and the main theorem of Hatcher \cite{MR0420620}, see Proposition \ref{prop:irreducible} below.  The difficulty in the reducible case lies in understanding the relationship between the topology of $\BDiff_0(M)$ and that of the classifying spaces of the diffeomorphism groups of the prime factors of $M$.  C\' esar de S\' a and Rourke (\cite{MR513752}) made a proposal to describe the homotopy type of $\Diff(M)$ in terms of the homotopy type of diffeomorphisms of the prime factors and an extra factor of the loop space on ``the space of prime decompositions". Hendriks-Laudenbach (\cite{MR780734}) and  Hendriks-McCullough (\cite{MR893801}) found a  precise model for this extra factor. Later Hatcher (\cite{HatcherUnfinished}) gave a finite dimensional model for this ``space of prime decompositions" and more interestingly, he proposed that there should be a ``wrong-way map" between $\BDiff(M)$ and the classifying space of diffeomorphisms of prime factors. Unfortunately, his approach was never completed.  \Cref{main} is inspired by his wrong-way map. But our proof of Theorem \ref{main} avoids some of the technical difficulties of Hatcher's proposed approach by combining $3$-manifold techniques with certain semi-simplicial techniques used in the parametrized surgery theory in the work of Galatius and  Randal-Williams (\cite{galatius2014homological}).  

The case $M \cong D^2 \times S^1$ is exceptional: in this case it is shown in \cite[Theorem 5.1]{gabai2001smale} that it is a consequence of Hatcher's theorem (\cite{hatcher1983proof}) that the restrict-to-boundary map $\Diff_0(M)\to \Diff_0(\partial M)$ is a weak equivalence. Therefore, the induced map $H^2(\BDiff_0(T);\bQ)\to H^2(\BDiff_0(M);\bQ)$ is an isomorphism.  
However, we can treat this case by giving a completely independent, dynamical rather than cohomological argument.  While it uses $C^1$ differentiability in an essential way, the strategy is general enough to apply both to $D^2 \times S^1$ and to any reducible 3-manifold with torus boundary.   This is carried out in section \ref{sec:dynamical}. 
Combining these results, we obtain the following answer to Ghys' problem.  

\begin{thm} \label{thm:torus_general} (Extending actions on the torus).
\begin{itemize}
\item Suppose $\rho: \Gamma \to \Diff_0(T^2)$ is an action that extends to a $C^0$ action on an orientable $3$-manifold $M$ with $\partial M\cong T^2$.  If the obstruction classes $\rho^*(x_1)$ and $\rho^*(x_2)$ are linearly independent in cohomology with rational coefficients, then $M \cong D^2 \times S^1$.  
\item There is a (explicitly given) finitely generated group $\Gamma$ such that, for any manifold with $\partial M \cong T^2$, there is an action $\Gamma \to \Diff_0(\partial M)$ that does not extend to the group of $C^1$-diffeomorphisms $\Diff^1_0(M)$.  
\end{itemize}  
\end{thm} 
As a concrete, simple example and special case of the first item in the above result, for a $3$-manifold $M$ with $\partial M\cong T^2$, not homeomorphic to $D^2 \times S^1$,  the $S^1 \times S^1$ action on the boundary does not lift to a $C^0$-action on $M$ (even discontinuously!).   

We also treat the case of manifolds with sphere boundary, again using a dynamical argument.
\begin{thm} \label{thm:sphere}
Let $M$ be an orientable 3-manifold with $\partial M \cong S^2$.  Then there is no extension $\Diff_0(\partial M) \to \Diff^1_0(M)$.  
\end{thm}
Note that if $M \cong B^3$ any group action on $S^2$ can be coned off to an action by homeomorphisms on the ball, thus the necessity of the differentiability hypothesis. Moreover, in this case, again by Hatcher's theorem (\cite{hatcher1983proof})  the restrict-to-boundary map $\Diff_0(M)\to \Diff_0(\partial M)$ is a weak equivalence. Therefore, in this case, there is no cohomological obstruction either. 

In the case where the extension is assumed continuous (giving the possibility of {\em topological} rather than purely algebraic obstructions), recent work of \cite{ChenMann} gives a negative answer to Ghys' original question in the smooth case, and in many settings in the $C^0$ case.  In some cases, continuity of group actions is known to be automatic \cite{Hurtado, Mann_automatic}, but even this is not enough to recover Theorems \ref{thm:torus_general} and \ref{thm:sphere} above.  
However, continuity motivates us to look for cohomological obstructions in the continuous setting.  In Section \ref{sec:p_1} we discuss this problem briefly for Pontyagin classes for manifolds with sphere boundary.

\paragraph{Further questions} 
Ghys' question can be related to the following notion of bordism of group actions. 

\begin{defn} \label{bordism def}
Let $N_1$ and $N_2$ be oriented $n$-manifolds, $\Gamma$ a discrete group, and $\rho_i: \Gamma \to \Diff(N_i)$ a homomorphism.  We say $\rho_1$ and $\rho_2$ are \emph{bordant} if there is a $(n+1)$-manifold $M$ and a representation $\phi: \Gamma \to \Diff(M)$ such that $\partial M = N_1 \sqcup -N_2$ and such that the restriction of $\phi(\gamma)$ to $N_i$ agrees with $\rho_i(\gamma)$ for each $\gamma \in \Gamma$.  
\end{defn}

For fixed $\Gamma$ and $n$, this notion of bordism gives an equivalence relation (standard techniques can be used to smooth a gluing of two actions that agree on a glued boundary), and bordism classes of group actions form a group $\Delta(n, \Gamma)$ under disjoint union.  This group is considered to be \emph{trivial} if it reduces to the ordinary (oriented) bordism group $\Omega_n$.   This setting generalizes both our extension problem (by not requiring diffeomorphisms to be isotopic to the identity) and Browder's notion of the bordism group $\Delta_n$ of diffeomorphisms of $n$-manifolds introduced in \cite{Browder}.  In our notation, $\Delta(n, \bZ)$ is Browders $\Delta_n$.   Similar definitions have appeared elsewhere in the literature, see for example \cite{Wasserman} for the case where $\Gamma$ is a compact Lie group.

The groups $\Delta(n, \bZ)$ have been computed for all $n$ by the combined work of Kreck, Melvin, Bonahon, and Edmonds-Ewing \cite{Kreck, Melvin, Bonahon, EdmondsEwing}.   To the best of our knowledge, this is the only case of a finitely generated, infinite group whose bordism groups are known, and known not to be trivial.  It seems to the authors that computing $\Delta(n, \Gamma)$ is beyond the techniques of this paper.  We instead propose two particularly interesting next cases for study.  

\begin{problem} 
Find nontrivial elements in $\Delta(1, \bZ \times \bZ)$.  
\end{problem} 
\begin{problem}
For the group $\Gamma$ defined in Proposition 4.1, find nontrivial elements in $\Delta(2, \Gamma)$.  
\end{problem} 

Both the extension and the bordism problem are already quite challenging in dimension 2.  We address the case of groups acting on tori and spheres focusing on $\Diff_0(M)$ here.  The extension problem for group actions on higher genus surfaces seems more difficult.  For instance, our dynamical approach in the torus case uses torsion elements, and $\Diff_0(S)$ is torsion free provided $S$ has genus at least 2 (see \cite{putmanaction} for a proof of this fact which is originally due to Hurwitz).   Furthermore, by Earle--Eells \cite{EarleEells}, the group $\Diff_0(S)$ is contractible, so there can be no cohomological obstructions to a continuous section (i.e obstruction classes in $\mathrm{B}\Diff_0(S))$.  However, it is possible that the cohomology of $\Diff_0(S)^\delta$, which is known to be nontrivial, could be used to give an obstruction class for the extension problem.  As a concrete instance, for any orientable surface $S$ there is a surjection of $H_3(\Diff_0^\delta(S); \bQ)$ to $\bR^2$ \cite{Bowden} (see also \cite{nariman2015stable}), the two (continuously varying) classes come from integrating Godbillon-Vey classes of foliations on flat bundles.   
\begin{problem}
For which, if any, $3$-manifolds with surface boundary do these ``Godbillon-Vey" classes provide obstructions to extending group actions on the boundary? 
\end{problem}

We also pose the following general problem.
\begin{problem}
Does there exist an example of a finitely generated group $\Gamma \subset \Diff_0(\Sigma_g)$, $g \geq 2$ so that the embedding of $\Gamma$ into $\Diff_0(\Sigma_g)$ is not nullbordant? 
\end{problem}

\subsection*{Acknowledgment}SN was partially supported by  NSF grant DMS-1810644 and he would like to thank the Isaac Newton Institute for Mathematical sciences for support and hospitality during the program ``homotopy harnessing higher structures".  He thanks Oscar Randal-Williams and S\o ren Galatius for helpful discussions. KM was partially supported by NSF grant DMS-1606254.
We also thank Sander Kupers for his comments on the early draft of this paper, and the referee for many helpful remarks.  

\section{Obstruction classes and a proof when $M$ is irreducible}\label{sec:prelim}

All manifolds, for the remainder of the paper, will be assumed smooth and orientable.  We assume basic familiarity with classifying spaces for topological groups, the reader may refer to \cite{MannTshishiku} for a very brief introduction in the context of related section problems for diffeomorphism groups, or \cite{morita2001geometry} for more detailed background.   

\subsection{Obstruction classes} 
Let $M$ be a manifold with boundary.  An extension $\phi$ of an action $\rho\colon \Gamma \to \Diff_0(\partial M)$ gives rise to a commutative diagram on the cohomology of the classifying spaces  

\begin{displaymath}
   \xymatrix @M=3pt {
          & H^*(\mathrm{B} \Diff_0(M)) \ar[dl]_{\phi^*} \\
        H^*(\mathrm{B} \Gamma)      & H^*(\mathrm{B} \Diff_0(\partial M)) \ar[l]^{\rho^* \hspace{.4cm}}  \ar[u]^{r^*}  }
\end{displaymath}

We define an \emph{obstruction class for $\rho$} to be any nonzero element of $\rho^*(\ker(r^*)) \subset H^*(\mathrm{B} \Gamma)$. 
It is immediate from the diagram above that if an extension $\phi$ of $\rho$ exists, then all obstruction classes vanish.   

To apply this in the setting of Theorem \ref{main}, we wish to find group homomorphisms
 $\rho\colon \Gamma\to \Diff_0(T)$ 
so that the induced map $\rho^*\colon H^*(\BDiff_0(T);\bQ)\to H^*(\mathrm{B}\Gamma;\bQ)$ is non-trivial on the generators of $H^2(\BDiff_0(T);\bQ)$. 
(We have used $\bQ$ coefficients here because it will be helpful much later in the proof; however, for the moment the reader may just as well work integrally.) 

It is a theorem of Earle--Eells \cite{EarleEells} that the inclusion of $\text{SO}(2) \times \text{SO}(2)$ into $\Diff_0(S^1 \times S^1)$ is a homotopy equivalence.   
Thus, $\BDiff_0(T^2) \simeq \bC P^\infty \times \bC P^\infty$ with cohomology generated by two classes in degree 2, corresponding to the Euler classes of each factor.  

Let $\Gamma =  \pi_1(\Sigma_g)$ be the fundamental group of a surface of genus $g \geq 2$.
For the standard embedding of $\Gamma$ as a lattice in $\mathrm{PSL}(2, \bR) \subset \Diff_0(S^1)$ (equivalently, the holonomy representation of the unit tangent bundle of $\Sigma_g$ equipped with a hyperbolic metric), the induced map on cohomology  
 $ H^*(\mathrm{B} \Diff_0(S^1);\bQ) \to H^*(\mathrm{B} \Gamma;\bQ)$ is not the zero map; indeed, as is well known, the pullback of a generator of $H^*(\BDiff_0(S^1);\bZ)$ (thought of under the standard inclusion into $H^*(\BDiff_0(S^1);\bQ)$)
 evaluated on the fundamental class of $H^*(B \Gamma;\bQ) \cong H^*(\Sigma_g;\bQ)$ gives the Euler characteristic of $\Sigma_g$.  Thus, these representations $\Gamma \to \Diff_0(S^1) \times \Diff_0(S^1) \to \Diff_0(T^2)$ via inclusion of $\Gamma$ into either $\Diff_0(S^1)$ factor give candidates for obstruction classes for extensions whenever $\partial M \cong T^2$.  
The proof of Theorem \ref{main} in the case $M$  is not diffeomorphic to $D^2 \times S^1$ now simply consists in showing that, for any such manifold $M $ with $\partial M \cong T^2$, there is an obstruction class in $H^2(\BDiff_0(T^2);\bQ)$.   

It is a well known fact in dimension $2$, and in dimension $3$ a theorem of Cerf (\cite{cerf1961topologie}) based on Smale's conjecture which was later proved by Hatcher (\cite{hatcher1983proof}), that the inclusion $\Diff(M)\hookrightarrow\mathrm{Homeo}(M)$ is a weak homotopy equivalence. Therefore, \Cref{main} also implies that the map
\[H^2(\mathrm{BHomeo}_0(T);\bQ)\to H^2(\mathrm{BHomeo}_0(M);\bQ),\]
sends one of the generators of $H^2(\mathrm{BHomeo}_0(T);\bQ)$ to zero, giving an obstruction to an extension by homeomorphisms.

\subsection{On the irreducible case}
As a warm-up and first case, we discuss the case where $M$ is irreducible, using the following result of Hatcher.  
\begin{thm}[Hatcher \cite{MR0420620}] \label{thm:hatcher}
If $M$ is an orientable, Haken 3-manifold which is not a closed Seifert manifold, then the group of diffeomorphisms that restrict to the identity on the boundary of $M$ has contractible components.
\end{thm} 

We prove the following.  
\begin{prop}  \label{prop:irreducible}
Let $P$ be an irreducible 3-manifold with $\partial P \cong T^2$, and assume that $P$ is not diffeomorphic to  $D^2 \times S^1$.  Then the map induced by the restriction to the boundary map 
\[H^2(\BDiff_0(T);\bQ)\to H^2(\BDiff_0(P);\bQ),\]
has a nontrivial kernel.

\end{prop}

\begin{proof}
Dually, it is enough to show that the boundary restriction map
\[
H_2(\BDiff_0(P);\bQ)\to H_2(\BDiff_0(T);\bQ),
\]
does not surject onto $H_2(\BDiff_0(T);\bQ)$.
Note that $\BDiff_0(P)$ and $\BDiff_0(T)$ are simply connected. Hence, by the Hurewicz theorem
\[
H_2(\BDiff_0(P);\bZ)\cong \pi_2(\BDiff_0(P))=\pi_1(\Diff_0(P)),
\]
\[
H_2(\BDiff_0(T);\bZ)\cong \pi_2(\BDiff_0(T))=\pi_1(\Diff_0(T))\cong \bZ^2.
\]
Therefore, it is enough to show that the map $\pi_1(\Diff_0(P))\to \pi_1(\Diff_0(T))$ is not surjective after tensoring with $\bQ$. Here we  have two cases:

\noindent {\bf Case 1}: For a fixed point $x\in T$, suppose $\pi_1(P, x)$ has nontrivial center. 
 Since $P$ is a prime manifold with torus boundary, it is Haken, and by a theorem of Waldhausen \cite{MR0236930} it is therefore Seifert fibered. 
By  the theorem of Hatcher (see \cite{hatcher1999spaces}) 
the group $\Diff_0(P)$  has the homotopy type of the circle, unless $P$ is diffeomorphic to $D^2\times S^1$, which is excluded by the hypothesis.  Therefore, $\pi_1(\Diff_0(P))\otimes \bQ\to \pi_1(\Diff_0(T))\otimes \bQ$ cannot be surjective.

\noindent{\bf Case 2:} Suppose $\pi_1(P, x)$ has trivial center. By considering the long exact sequence for the homotopy groups of the fibration $\Diff(P,\partial)\to \Diff(P)\to \Diff(T)$, it is enough to show that the map
\[
\pi_1(\Diff(T))\to \pi_0(\Diff(P,\partial)),
\]
sends at least one of the generators to a non-torsion mapping class. 

To show that a Dehn twist around the boundary is non-trivial, we look at its action on $\pi_1(P, x)$. This action is given by the conjugation of the loops on the boundary torus. If $\pi_1(P, x)$ has no center, then these Dehn twists are non-trivial in the mapping class group. To show that the nontrivial mapping class induced by the Dehn twist around a generator of the boundary is non-torsion, we show that its conjugation action on $\pi_1(P, x)$ is non-torsion. To do so, it is enough to show that the map
\[
\pi_1(T,x)\to \pi_1(P,x),
\]
 is in fact injective. If there is a non-trivial kernel, the loop theorem \cite{Papakyriakopoulos} implies that there is  a simple closed curve on $T$ that bounds a properly embedded disk $D$ in $P$. But now the union of $D$ and $T$ gives an embedded sphere in $P$ and since $P$ is irreducible, this sphere has to bound a ball. Therefore $P$ would be diffeomorphic to $D^2\times S^1$ which contradicts the hypothesis.
\end{proof}

With this argument for the irreducible case in hand, one can obtain Theorem \ref{thm:torus_general} for extensions to $\Diff^1(M)$ 
with a short dynamical argument.  The dynamical argument is given in Section \ref{sec:dynamical}, and can be read independently from Section \ref{sec:main_proof}.    However, for the moment we continue with the cohomological approach, building towards a proof of \Cref{main}.


\section{Proof of Theorem \ref{main}}  \label{sec:main_proof}


The broad strategy of this proof is to use semi-simplicial spaces that parametrize different ways of cutting $M$ along separating spheres, motivated by the desire to reduce the situation to the irreducible case above.  If $S$ is an embedded sphere in $M$ that separates $P$, then the pointwise stabilizer $\text{Stab}(S) \subset \Diff_0(M)$ consisting of diffeomorphisms that are the identity on $S$, sits in a zig-zag
 \[
\Diff_0(P\backslash\text{int}(D^3))\xleftarrow{\text{res}} \text{Stab}(S)\hookrightarrow \Diff_0(M),
 \]
where the left map is the restriction map. In fact, for any separating sphere $S$, we have the map
\[
\mathrm{B}\text{Stab}(S)\to \BDiff_0(T),
\]
induced by the restriction to the boundary. We first use Proposition \ref{prop:irreducible} to prove that 
\[
H^2(\BDiff_0(T);\bQ)\to H^2(\mathrm{B}\text{Stab}(S);\bQ),
\]
has a nontrivial kernel. Using the semi-simplicial techniques and a spectral sequence argument, we then prove that for a non-irreducible $M$, the map 
\[
H^2(\BDiff_0(T);\bQ)\to H^2(\BDiff_0(M);\bQ),
\]
also has a nontrivial kernel.

\subsection{Semi-simplicial resolution for $\BDiff_0(M)$} We want to make an inductive argument by cutting $M$ into factors with fewer prime factors. To do so, we first define an auxiliary  simplicial complex of {\em sphere systems} from which we construct a semi-simplicial space on which $\Diff_0(M)$ acts.
 As these play a key role in the proof, the reader unfamiliar with semisimplicial spaces may wish to consult \cite{ebert2017semi} for an introduction. 
\begin{defn}
For a $3$-manifold $M$, a {\em sphere system} $s$ consists of a finite collection of disjoint parametrized essential (i.e. not bounding a ball)
 spheres in $M$ such that as we cut $M$ along the spheres in $s$, the connected components are either prime manifolds with disjoint balls removed or $S^3$ with disjoint balls removed. A sphere system is allowed to have parallel spheres or spheres that are isotopic to the sphere boundary components of $M$. 
\end{defn}
\begin{defn}
The sphere complex $\mathcal{S}(M)$ is the simplicial complex whose vertices are sphere systems in $M$ and a $p$-simplex is given by $p+1$ disjoint sphere systems. 
\end{defn}
We shall define a semisimplicial space using $\mathcal{S}(M)$ as follows.
\begin{defn}
Let $X_{\bullet}(M)$ be a semisimplicial space whose space of $0$-simplices $X_0(M)$ has the same underlying set as the set of vertices of $\mathcal{S}(M)$. We topologize $X_0(M)$ as the subspace of configuration space of spheres in $M$
\[
\coprod_n \text{Emb}(S^2,M)^n/\Sigma_n,
\]
where $\text{Emb}(S^2,M)$ is the space of smooth embeddings with  $C^{\infty}$-topology and $\Sigma_n$ is  the permutation group on $n$ letters that permutes the spheres. The space of $p$-simplices $X_p(M)$ is the subspace of $X_0(M)^{p+1}$ given by $(p+1)$-tuples of disjoint sphere systems. The $i$-th face maps are given by omitting the $i$-th sphere system. 
\end{defn}
Our goal in this section is to prove that when $M$ is not prime, the realization $|X_{\bullet}(M)|$ is weakly contractible. We first show that $\mathcal{S}(M)$ is a contractible simplicial complex whenever it is nonempty. 

\begin{lem}\label{claim2}
Let $M$ be a $3$-manifold which is not  prime, then the simplicial complex $\mathcal{S}(M)$ is contractible.
\end{lem}
\begin{proof}
We want to show that for all $k$, any continuous map $f\colon S^k \to \mathcal{S}(M)$ is nullhomotopic.  Without loss of generality, we can assume that for  a triangulation $K$ of $S^k$,  the map $f$ is piecewise linear. To find a nullhomotopy for the map $f$, it is enough to homotope it so that its image lies in the star of a vertex in $\mathcal{S}(M)$.

Recall that each vertex in $f(K)$ is a sphere system, with higher-dimensional simplices given by disjoint sphere systems. By the transversality theorem, we can slightly perturb $f$ so that sphere systems represented by the vertices of $f(K)$ are pairwise transverse. Let $w\in \mathcal{S}(M)$ be a sphere system that  is transverse to each of the sphere systems represented by the set of vertices in $f(K)$.  
We will show how to produce a homotopy of $f$ which decreases the (finite) number of circles in the intersection of $w$ and the spheres in $f(K)$.  Applying this procedure iteratively gives a homotopy of $f$ to a map with image in  $\text{Star}(w)$. 

The intersections of spheres in $w$ with the vertices of $f(K)$ give a collection of circles on the spheres of $w$.  From this collection, choose a maximal family of disjoint circles, and let $C$ be an innermost circle in this family.  Then $C$ is given by the intersection of a sphere $S(v)$ in a vertex $v = f(x)\in f(K)$ and a sphere $S(w)$ in the system $w$, and {\em innermost} means that $C$ bounds a disk $D$ on $S(w)$ whose interior is disjoint from all spheres in the maximal collection.
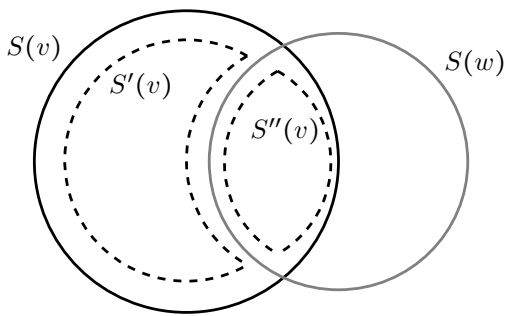
\begin{figure}[h] 

 \begin{tikzpicture}[node distance=6cm]
 \draw [line width=1.05pt] (0,0) circle (2cm);
  \draw [gray, line width=1.05pt] (2,0) circle (1.7cm);
\draw [dashed, line width=1.05pt] (0.8,1.4) arc (60:300:1.6cm);
\draw [dashed, line width=1.05pt] (0.8,1.4) arc (120:240:1.6cm);
\draw [dashed, line width=1.05pt] (1.2,1.2) arc (120:240:1.4cm);
\draw [dashed, line width=1.05pt] (1.2,1.2) arc (60:-60:1.4cm);
\node at (-2, 1.5) {$S(v)$};
\node at (-0.6, 1) {$S'(v)$};
\node at (1.3, 0.4) {$S''(v)$};
\node at (3.8, 1.3) {$S(w)$};
 \end{tikzpicture}
 \caption{Surgery on spheres in one dimension lower}\label{surgery} \end{figure} 
 
We can cut $S(v)$ along the circle $C$  and glue two copies of this $2$-disk to obtain two disjoint embedded spheres $S'(v)$ and $S''(v)$ (see \Cref{surgery}). We parametrize these spheres arbitrarily. By considering nearby parallel copies, we can assume that $S(v)$, $S'(v)$ and $S''(v)$ are disjoint. Note that at least one of the spheres $S'(v)$ and $S''(v)$ is essential (i.e. does not bound a ball).  Now replace $S(v)$ by the two spheres $S'(v)$ and $S''(v)$ if both $S'(v)$ and $S''(v)$ are separating,  and if just one of them is separating we replace $S(v)$ with that one. In this way, we obtain a new vertex $v'\in \mathcal{S}(M)$. By choosing nearby parallel copies of the spheres, we can assume that the vertex $v'$ is adjacent to $v$ i.e. their corresponding sphere systems are disjoint.

If we choose $S'(v)$ and $S''(v)$ sufficiently close to the $2$-disk in $S(v)$ that bounds $C$, then any sphere $S$ in the sphere systems in the star of $v$ which intersected $S'(v)$ or $S''(v)$ would also intersect this $2$-disk.  However, this cannot happen: since $S \cap S(w) = \emptyset$, and $C$ was chosen to be an innermost circle among a maximal family of disjoint circles given by intersections with $S(w)$, no disjoint sphere $S$ in the sphere systems in the star of $v$ can intersect the $2$-disk bounded by $C$.  Thus, no vertex in the star of $v$ intersects $S'(v)$ and $S''(v)$, so our modified sphere system  $v'$ 
 remains disjoint from all the sphere systems that $v$ is disjoint from.  In other words, $v'$ is adjacent to all vertices in the star of $v$. Therefore, we have a simplicial homotopy $F\colon K\times [0,1]\to \mathcal{S}(M)$ such that $F(-,1)$ is the same as $F(-,0)$ on all vertices but $x$ and $F(x,1)=v'$. Note that the vertices in the image $F(-,1)\colon K\to \mathcal{S}(M)$  have fewer circles in their intersection with $S(w)$. By repeating this process for all spheres in the sphere system $w$, we could homotope the map $f$ to a map whose image lies in the star of $w$. Therefore, $f$ is nullhomotopic.
\end{proof}

\begin{rem} \label{weakly_CM}
Note that the same argument implies that the link $\text{Lk}_{\sigma}\mathcal{S}(M)$ of a $p$-simplex $\sigma$ in $\mathcal{S}(M)$ is also contractible. Because as we cut along the sphere systems in $\sigma$ we obtain union of $3$-manifolds with sphere boundaries. Hence the sphere complex of each piece is contractible. Therefore, $\text{Lk}_{\sigma}\mathcal{S}(M)$ which is the join of the sphere complex of pieces is also contractible.
\end{rem}
 
A complex which is not only contractible but has the property that the link of each simplex is also contractible is called  {\it weakly Cohen-Macaulay} of dimension infinity.  
Now we use the {\it generalized coloring lemma} (\cite[Theorem 2.4]{galatius2014homological}) for such complexes to prove the following.

\begin{prop}\label{prop1} The realization $|X_{\bullet}(M)|$ is weakly contractible.
\end{prop}
\begin{proof} Let $X_{\bullet}^{\delta}(M)$ be the underlying semisimplicial set of the semisimplicial space $X_{\bullet}(M)$. Using the powerful ``discretization" technique  from \cite[Theorem 5.6]{galatius2014homological}, the contractibility of the realization $|X_{\bullet}^{\delta}(M)|$ implies the weak contractibility of $|X_{\bullet}(M)|$.

The simplices in $\mathcal{S}(M)$ do not have a natural ordering on their vertices. For each ordering of vertices of a simplex in $\mathcal{S}(M)$, there is a corresponding simplex in the $\Delta$-complex $|X_{\bullet}^{\delta}(M)|$. But since $\mathcal{S}(M)$ is a weakly Cohen-Macaulay complex of dimension infinity by Remark \ref{weakly_CM}, the proofs of \cite[Proposition 2.10]{MR3677942} and also \cite[Theorem 3.9]{nariman2014homologicalstability} apply and show that the contractibility of $|X_{\bullet}^{\delta}(M)|$ follows from contractibility of $\mathcal{S}(M)$ using the {\it generalized coloring lemma} (\cite[Theorem 2.4]{galatius2014homological}).
\end{proof}
 
 The next step in the proof is to use the action of $\Diff_0(M)$ on the semisimplicial space $X_{\bullet}(M)$ to find a semisimplicial resolution for $\BDiff_0(M)$. But given that two different sphere systems are not necessarily isotopic, the action of $\Diff_0(M)$ on $X_{\bullet}(M)$ is not transitive and in fact it is not clear how to describe the set of the orbits of this action.  This creates a technical issue for us, as understanding the orbits will be useful in analyzing the spectral sequence for semisimplicial resolutions. To get around this, we first define a larger group $\SDiff(M)$ generated by the {\em slide diffeomorphisms} that contains $\Diff_0(M)$.  As we shall see, the spectral sequence for the action of this group is easier to study and it will be sufficient to prove \Cref{Main}.
 
 \subsubsection{Slide diffeomorphisms} 
 McCullough in \cite[\S 3]{MR925856} showed that the mapping class group of a compact orientable $3$-manifold $M$ is generated by four types of mapping classes. Let $S=\coprod S_i$ be a special sphere system as follows.
 
 \begin{defn}\label{spheresystems}
Let $S$ be a collection of disjoint smooth embeddings $\phi\colon S^2\hookrightarrow M$ of separating spheres. Let $M_0$, $M_1$, \dots, $M_k$, $M_{k+1}$,\dots, $M_{k+g}$ be the components of the manifold obtained from $M$ by cutting it along $S$ where $M_{k+i}$ is diffeomorphic to $S^2\times [0,1]$ for all $i>0$. Let $\hat{M_i}$ be the manifold obtained from $M_i$ by gluing a ball to every sphere boundary component. We say $S$ is a {\em special sphere system} if
\begin{itemize}
\item $\hat{M_i}$ is an irredicible manifold for all $i\leq k$.
\item $\hat{M_0}$ is diffeomorphic to $S^3$. 
\item For $1\leq i\leq k$, the manifold $\hat{M_i}$ is {\it not} diffeomorphic to $S^3$ and $M_i$ has exactly one sphere boundary component. 
\end{itemize}
\end{defn}
 
Let $M_i(S)$ be the components obtained from $M$ by cutting along $S$. Following McCullough, every diffeomorphism of $M$ is isotopic to the composition of diffeomorphisms of the following types
 \begin{enumerate}
 \item {\it Diffeomorphisms of the factors}. This is the subgroup of diffeomorphisms that restricts to the identity on $M_0(S)$; it is isomorphic to the product over all $i$ of $\Diff(M_i(S), \partial M_i(S))$. 
 \item {\it Permuting diffeomorphic factors}. If two factors $M_i(S)$ and $M_j(S)$ are diffeomorphic, we have elements in $\Diff(M)$ that leave $M_0(S)$ invariant, interchange $M_i(S)$ and $M_j(S)$ and restrict to the identity on the other factors.
 \item {\it Spinning factors that are diffeomorphic to $S^2\times [0,1]$}. For the factors $M_{k+i}(S)$ that are diffeomorphic to $S^2\times [0,1]$, we have an element of $\Diff(M)$ that leaves $M_0(S)$ invariant, interchanges the boundaries of $M_{k+i}(S)$, restricts to an orientation preserving diffeomorphism of $M_{k+i}(S)$ and restricts to the identity on the other factors.
 \item {\it Slide diffeomorphisms}. These diffeomorphisms slide a factor $M_i(S)$ for $i\leq k$ around an arc $\alpha$ in $M$ that intersects $M_i(S)$ only at its endpoints. To be more precise, let $\hat{M}$ be the manifold obtained by gluing a ball $B$ to $M\backslash \text{int}(M_i(S))$ and let $\alpha$ be an arc in $M\backslash \text{int}(M_i(S))$ that intersects $\partial M_i(S)$ at its end points. There is a disk pushing isotopy $h_t$ of $\hat{M}$ where $h_0=\text{id}$ and $h_1|_B=\text{id}$ so that $h_t$ moves $B$ along the arc $\alpha$. A {\it slide diffeomorphism} that slides $M_i(S)$ along $\alpha$ is a diffeomorphism $f\in \Diff(M)$ so that $f|_{M_i(S)}=\text{id}$ and on $M\backslash \text{int}(M_i(S))$, the diffeomorphism $f$ is equal to $h_1$. 
 \end{enumerate}

\begin{defn}
Let $\SDiff(M)$ be subgroup of $\Diff(M)$ that is generated by  slide diffeomorphisms. 
\end{defn}
Note that the ``restrict to boundary'' map from $\SDiff(M)$ also has image equal to $\Diff_0(T)$. Therefore, we have the homotopy commutative diagram 
\begin{equation}
\begin{gathered}
\begin{tikzpicture}[node distance=1.5cm, auto]
  \node (A) {$\BDiff_0(M)$};
  \node (B) [right of=A, below of=A, node distance=1.5cm]{$\BDiff_0(T).$};
  \node (C) [right of=A, node distance=3cm]{$\mathrm{B}\SDiff(M)$};
  \draw [->] (A) to node {$$} (C);
    \draw [<-] (B) to node {$r_s$} (C);
  \draw [->] (A) to node {$r$} (B);
 \end{tikzpicture}
 \end{gathered}
\end{equation}
Hence, to prove \Cref{main}, it is enough to prove the following theorem.
\begin{thm}\label{Main}
Let $M$ be an orientable three-manifold, not diffeomorphic to $D^2 \times S^1$, with $\partial M = T^2$. Then the induced map
\[r_s^*\colon H^2(\BDiff_0(T);\bQ)\to H^2(\mathrm{B}\SDiff(M);\bQ),\]
has a nontrivial kernel.
\end{thm}
 To prove this theorem, we use the homotopy quotient \footnote{For a topological group $G$ acting on a topological space $X$, the homotopy quotient is denoted by $X\hcoker G$ and is given by $X\times_G \mathrm{E}G$ where $\mathrm{E}G$ is a contractible space on which $G$ acts freely and properly discontinuously.} of the action of $\SDiff(M)$ on $X_{\bullet}(M)$ to define a semisimplicial resolution for $\mathrm{B}\SDiff(M)$. Let the map $\alpha$ be given by
\begin{equation}\label{resolution}
\alpha\colon X_{\bullet}(M)\hcoker \SDiff(M)\xrightarrow{} \mathrm{B}\SDiff(M).
\end{equation}
Since $X_{\bullet}(M)$ is a subspace of a product of embedding spaces, it is compactly generated weak Hausdorff space. Therefore, by \cite[Lemma 2.1]{randal2009resolutions}, the map $\alpha$ is a locally trivial fiber bundle with fibers homeomorphic to the geometric realization $|X_{\bullet}(M)|$. Given that $|X_{\bullet}(M)|$ is contractible by \Cref{prop1}, the map $|\alpha|$ between $|X_{\bullet}(M)\hcoker \SDiff(M)|$ and $\mathrm{B}\SDiff(M)$ is a weak equivalence. 
\subsection{Proof of \Cref{Main}} We have a homotopy commutative diagram 

\begin{equation}\label{311}
\begin{gathered}
\begin{tikzpicture}[node distance=1.5cm, auto]
  \node (A) {$X_{\bullet}(M)\hcoker \SDiff(M)$};
  \node (B) [right of=A, below of=A, node distance=1.5cm]{$\BDiff_0(T).$};
  \node (C) [right of=A, node distance=3cm]{$\mathrm{B}\SDiff(M)$};
  \draw [->] (A) to node {$$} (C);
    \draw [<-] (B) to node {$r_s$} (C);
  \draw [->] (A) to node {$\beta_{\bullet}$} (B);
 \end{tikzpicture}
 \end{gathered}
\end{equation}
Recall  that our goal is to show that there exists a generator $x\in H^2(\BDiff_0(T);\bQ)$ so that $r_s^*(x)=0$. As a first step, we show that for all $p$, the class $\beta_p^*(x)$ vanishes in $H^2(X_{p}(M)\hcoker \SDiff(M);\bQ)$.

\begin{lem}\label{prop:20} Let $M$ be a $3$-manifold that bounds a torus and is not diffeomorphic to $D^2\times S^1$. There exists a generator $x\in H^2(\BDiff_0(T);\bQ)$ such that for each $p$, the class $x$ is in the kernel of the map induced by $\beta_p$
\[
H^2(\BDiff_0(T);\bQ)\to H^2(X_{p}(M)\hcoker \SDiff(M);\bQ).
\]
\end{lem}
  
\begin{proof}
First, we shall describe the homotopy type of $X_{p}(M)\hcoker \SDiff(M)$ in terms of stabilizers of the action of $\SDiff(M)$ on $X_p(M)$. For a $p$-simplex $\sigma_p\in X_p(M)$, let $\text{Stab}(\sigma_p)$ be the stabilizer of $\sigma_p$ under the action of $\SDiff(M)$. The isotopy extension theorem implies that we have a fibration 
\[
\text{Stab}(\sigma_p)\to \SDiff(M)\to X_p(M),
\]
where the last map is the evaluation map on $\sigma_p$. In fact the local triviality (see \cite[Remark page 307]{MR0123338}) of the evaluation map implies that $X_p(M)$ is homeomorphic to $\SDiff(M)/\text{Stab}(\sigma_p)$. Therefore, the natural map
\begin{equation}\label{f}
f\colon\mathrm{B}\text{Stab}(\sigma_p)\to X_{p}(M)\hcoker \SDiff(M),
\end{equation}
is a weak equivalence. Thus it is enough to show that there exists a generator $x\in H^2(\BDiff_0(T);\bQ)$ that lies in the kernel of the map 
\[
H^2(\BDiff_0(T);\bQ)\xrightarrow{f^*} H^2(\mathrm{B}\text{Stab}(\sigma_p);\bQ).
\]
Let $M_i(\sigma_p)$ denote the components of the manifold obtained from $M$ by cutting along sphere systems in $\sigma_p$. Note that if a slide diffeomorphism $f$ fixes $M_i(\sigma_p)$ setwise, it will lie in $\Diff_0(M_i(\sigma_p))$ (i.e. its restriction to $M_i(\sigma_p)$ is isotopic to the identity in the group $\Diff(M_i(\sigma_p))$ of diffeomorphisms that preserve, but do not necessarily pointwise fix the boundary). Let $P\backslash \text{int}(D^3)$ be the connected component containing the torus boundary when we cut $M$ along the embedded spheres in the $p$-simplex $\sigma_p$.  We have a homotopy commutative diagram 
\[
\begin{tikzpicture}[node distance=1.5cm, auto]
  \node (A) {$\mathrm{B}\text{Stab}(\sigma_p)$};
  \node (B) [right of=A, below of=A, node distance=1.5cm]{$\BDiff_0(T).$};
  \node (C) [right of=A, node distance=3.3cm]{$\BDiff_0(P\backslash \text{int}(D^3))$};
  \draw [->] (A) to node {$\text{res}$} (C);
    \draw [<-] (B) to node {$g$} (C);
  \draw [->] (A) to node {$f$} (B);
 \end{tikzpicture}
\]
Thus, it is enough to show that $g^*(x)=0$ for a generator $x$. We consider two different cases depending on whether $P$ is diffeomorphic to $D^2\times S^1$. 

\noindent{\bf Case 1}: Suppose $P$ is diffeomorphic to $D^2\times S^1$. Dually, it suffices to show that the map
\[
H_2(\BDiff_0(P\backslash \text{int}(D^3));\bQ)\to H_2(\BDiff_0(T);\bQ),
\]
is not surjective. By the Hurewicz theorem, it is enough to prove  that the map $\pi_1(\Diff_0(P\backslash \text{int}(D^3)))\to \pi_1(\Diff_0(T))$ does not hit both generators of $\bZ^2$. 

Let $\Diff_0(P\backslash \text{int}(D^3), \partial_{\text{SO}(3)})$ be the subgroup of $\Diff_0(P\backslash \text{int}(D^3))$ consisting of those diffeomorphisms that restrict to a rotation on the parametrized sphere boundary. Because $\Diff_0(S^2)\simeq \text{SO}(3)$, the inclusion 
\[
\Diff_0(P\backslash \text{int}(D^3), \partial_{\text{SO}(3)})\xrightarrow{\simeq} \Diff_0(P\backslash \text{int}(D^3)),
\]
is a weak equivalence.

Moreover, the group $\Diff(P\backslash \text{int}(D^3), \partial_{\text{SO}(3)})$ sits in a fiber sequence
\[
\Diff(P\backslash \text{int}(D^3), \partial_{\text{SO}(3)})\to \Diff(P)\to \text{Emb}^{\text{fr}}(D^3, P),
\]
where $\text{Emb}^{\text{fr}}(D^3, P)$ is the space of framed embeddings of a $3$-ball into $P$.  It is homotopy equivalent to $P \cong D^2\times S^1$.  Thus, from the long exact sequence of homotopy groups, we obtain
\[
0\to \pi_1(\Diff_0(P\backslash \text{int}(D^3)))\xrightarrow{\theta} \pi_1(\Diff(D^2\times S^1))\xrightarrow{\alpha} \pi_1(D^2\times S^1).
\]
Note that $\pi_1(\Diff(D^2\times S^1))\cong \bZ^2$ and $\pi_1(D^2\times S^1)\cong \bZ$ and the map $\alpha$ is the projection to the second factor. Therefore, the map $\theta$ does not hit both generators.

\noindent{\bf Case 2}: Suppose $P$ is not diffeomorhic to $D^2\times S^1$.  Since rotations on the sphere $S^2$ can be extended to diffeomorphisms of the $3$-ball, the group $\Diff_0(P\backslash \text{int}(D^3), \partial_{\text{SO}(3)})$ embeds into $\Diff_0(P)$. Therefore from the zig-zag of maps
\[
\Diff_0(P)\hookleftarrow\Diff_0(P\backslash \text{int}(D^3), \partial_{\text{SO}(3)})\xrightarrow{\simeq} \Diff_0(P\backslash \text{int}(D^3)),
\]
 we obtain the  commutative diagram
\[
\begin{tikzpicture}[node distance=1.5cm, auto]
  \node (A) {$H^2(\mathrm{B}\text{Stab}(\sigma_p);\bQ)$};
  \node (B) [right of=A, below of=A, node distance=1.5cm]{$H^2(\BDiff_0(T);\bQ).$};
  \node (C) [right of=A, node distance=3.6cm]{$H^2(\BDiff_0(P);\bQ)$};
  \draw [<-] (A) to node {$$} (C);
    \draw [->] (B) to node {$$} (C);
  \draw [<-] (A) to node {$f^*$} (B);
 \end{tikzpicture}
\]

\Cref{prop:irreducible} now implies that $f^*(x)=0$.
\end{proof}
\begin{rem}
Note that the proof of \Cref{prop:20} also implies that the generator $x$ only depends on the prime factor $P$ that contains the torus boundary component.
\end{rem}
To conclude the proof of \Cref{Main}, we use a spectral sequence argument.  
 Recall that for any semi-simplcial space $Y_{\bullet}$, there is a spectral sequence induced by the skeletal filtration on $|Y_{\bullet}|$
\begin{equation}\label{ssss}
\begin{tikzcd}
E^1_{p,q}(Y_{\bullet})=H^q(Y_{p};\bQ)\arrow[Rightarrow]{r}&H^{p+q}(|Y_{\bullet}|;\bQ),
\end{tikzcd}
\end{equation}
and the first differential $d^1\colon E^1_{p,q}(Y_{\bullet})\to E^1_{p+1,q}(Y_{\bullet})$ is given by the alternating sum of maps induced by the face maps (see \cite{segal1968classifying, ebert2017semi}). 

Since $X_{\bullet}(M)\hcoker \SDiff(M)$ is a semi-simplicial resolution for $\mathrm{B}\SDiff(M)$, the spectral sequence computing the cohomology of $|X_{\bullet}(M)\hcoker\SDiff(M)|$ takes the form
\begin{equation}\label{sss}
\begin{tikzcd}
E^1_{p,q}(X_{\bullet}(M)\hcoker \SDiff(M))=H^q(X_{p}(M)\hcoker \SDiff(M);\bQ)\arrow[Rightarrow]{r}&H^{p+q}(\mathrm{B}\SDiff(M);\bQ).
\end{tikzcd}
\end{equation}
Recall that we want to prove that $r^*(x)=0\in H^2(\mathrm{B}\SDiff(M);\bQ)$ in the diagram \ref{311}.  Denote the filtration on  $H^2(\mathrm{B}\SDiff(M);\bQ)$ in the above spectral sequence by
{\small\[
0\subseteq F_2H^2(\mathrm{B}\SDiff(M))\subseteq F_1H^2(\mathrm{B}\SDiff(M))\subseteq F_0H^2(\mathrm{B}\SDiff(M))=H^2(\mathrm{B}\SDiff(M);\bQ).
\]}
A priori $r^*(x)\in F_0H^2(\mathrm{B}\SDiff(M))$, but since by \Cref{prop:20}, we know $\beta_0^*(x)=0$ (in fact $\beta_p^*(x)=0$ for all $p$), the class $r^*(x)$ lives in the kernel of the natural map
\[
H^2(\mathrm{B}\SDiff(M);\bQ)\to H^2(X_{0}(M)\hcoker \SDiff(M);\bQ).
\]
 Hence $r^*(x)\in F_1H^2(\mathrm{B}\SDiff(M))$. Now we shall prove that the first row in the spectral sequence \ref{sss} vanishes. Therefore, in fact we have $r^*(x)\in F_2H^2(\mathrm{B}\SDiff(M))$.
\begin{lem}\label{firstrow}
The first row of the spectral sequence \ref{sss} vanishes, i.e. for all $p$ we have $H^1(X_{p}(M)\hcoker \SDiff(M);\bQ)=0$.
\end{lem}
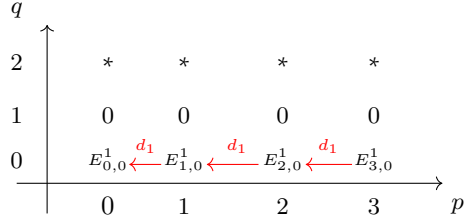
\begin{figure}
\[
\begin{tikzpicture}
\draw [<-] (-0.6,2.7) -- (-0.6,0.2);
\draw [->] (-1,0.6)-- (4.6,0.6);
\node  at (0.2,0.9) {\tiny$E^1_{0,0}$};
\node  at (0.2,0.3) {$0$};
\node at (1.2,0.9) {\tiny$E^1_{1,0}$};
\node at (1.2,0.3) {\small$1$};

\node at (2.5,0.9) {\tiny$E^1_{2,0}$};
\node at (2.5,0.3) {\small$2$};

\node at (3.7,0.9) {\tiny$E^1_{3,0}$};
\node at (3.7,0.3) {\small$3$};



\node at (4.8,0.3) {\small$p$};

\node at (0.2,1.5) {\small$0$};
\node at (-1,1.5) {\small$1$};
\node at (-1,0.9) {\small$0$};

\node at (0.2,2.2) {\small$*$};
\node at (-1,2.2) {\small$2$};



\node at (0.2,5.2) {\small$$};
\node at (-1,2.9) {\small$q$};

\node at (1.2,2.2) {\small$*$};
\node at (2.5,2.2) {\small$*$};
\node at (3.7,2.2) {\small$*$};

\node at (1.2,1.5) {\small$0$};
\node at (2.5,1.5) {\small$0$};
\node at (3.7,1.5) {\small$0$};
\draw [<-,red] (0.48,0.85)--(0.9,.85);
\draw [<-,red] (1.5,0.85)--(2.18,.85);
\draw [<-,red] (2.8,0.85)--(3.41,.85);

\node at (0.7,1.1) {\tiny$\textcolor{red}{d_1}$};
\node at (1.9,1.1) {\tiny$\textcolor{red}{d_1}$};
\node at (3.1,1.1) {\tiny$\textcolor{red}{d_1}$};

\end{tikzpicture}\]
\caption{ The first page of the homology spectral sequence calculating $H_*(|X_{\bullet}(M)\hcoker \SDiff(M)|;\bQ)$}\label{s'}
\end{figure}
\begin{proof}
Using the weak equivalence \ref{f} and the universal coefficient theorem, it is enough to show that for any simplex $\sigma$, we have
\[
H_1(\mathrm{B}\text{Stab}(\sigma);\bQ)=\pi_1(\mathrm{B}\text{Stab}(\sigma))_{\text{ab}}\otimes \bQ=0.
\]
Since $\pi_1(\mathrm{B}\text{Stab}(\sigma))=\pi_0(\text{Stab}(\sigma))$, we shall prove that $\pi_0(\text{Stab}(\sigma))$ is in fact a torsion group.

To do so, we shall freely pass to its finite index subgroups. Similar to \cite[Section 3]{MR1070575}, let $\mathcal{R}(M)$ be the subgroup of $\pi_0(\Diff(M))$ generated by the Dehn twists around embedded $2$-spheres in $M$. By \cite[Lemma 3.2]{MR1070575}, the group $\mathcal{R}(M)$ is a finite normal subgroup. Let $\overline{\pi_0(\Diff(M))}$ denote the quotient
\[
\pi_0(\Diff(M))/\mathcal{R}(M).
\]
For a simplex $\sigma$, let $M_i(\sigma)$ denote the components of the manifold obtained from $M$ by cutting along sphere systems in $\sigma$. By \cite[Lemma 3.4]{MR1070575}, we know that 
\[
\prod_i \overline{\pi_0(\Diff(M_i(\sigma), \partial))}\to \overline{\pi_0(\Diff(M))},
\]
is injective (this can also be seen using \cite[Proposition 2.1]{hatcher2010stabilization}). Hence, $\overline{\pi_0(\text{Stab}(\sigma))}$ is a subgroup of $\prod_i \overline{\pi_0(\Diff(M_i(\sigma), \partial))}$. On the other hand, using the definition of slide diffeomorphisms, one can see that if we restrict  a slide diffeomorphism in $\text{Stab}(\sigma)$ to $M_i(\sigma)$, its image in  $\overline{\pi_0(\Diff(M_i(\sigma), \partial))}$ is trivial. Therefore, $\overline{\pi_0(\text{Stab}(\sigma))}$ is trivial which implies that $\pi_0(\text{Stab}(\sigma))$ is a finite group.
\end{proof}

Since the first row of the spectral sequence is zero, we have  $$r^*(x)\in F_2H^2(\mathrm{B}\SDiff(M))= E^{2,0}_{\infty}(X_{\bullet}(M)\hcoker \SDiff(M))=E^{2,0}_2(X_{\bullet}(M)\hcoker \SDiff(M)).$$ 
Hence to show $r^*(x)=0$ it is enough to prove that $$F_2H^2(\mathrm{B}\SDiff(M);\bQ)=0.$$
To do so, in fact we prove a stronger result  that the $0$-row of this spectral sequence vanishes at $E_2$-page. In other words, the $0$-th row of the $E_1$-page is acyclic.
\begin{lem}\label{0th row}
The cochain complex $(E^{*,0}_1(X_{\bullet}(M)\hcoker \SDiff(M)),d^1)$ is acyclic.
\end{lem}
Before proving this lemma, we shall describe how to think of the set of the orbits of the action of $\SDiff(M)$ on $X_{\bullet}(M)$.
\subsection{On the orbits of the action of $\SDiff(M)$ on $X_{\bullet}(M)$} We shall prove using a construction essentially due to Scharlemann (see \cite[Appendix A, Lemma A.1]{Bonahon}), that each orbit has a representative inside a submanifold of $M$ that is diffeomorphic to $S^3$ with disjoint balls removed.

Let us fix $S=\coprod S_i$ to be a special sphere system as is defined in \Cref{spheresystems}. We denote the corresponding components of the manifold obtained from $M$ by cutting it along $S$ by $M_i(S)$. Then we have the following lemma (see \cite[Lemma 2.1]{MR1070575}).

\begin{lem}[Hatcher-McCullough] \label{HM}Let $S'$ be any sphere system. Then there exists an element $f$ of $\SDiff(M)$ such that $f(S')\subset M_0(S)$. 
\end{lem}
\begin{proof}[Proof sketch] Bonahon in \cite[Appendix A, Lemma A.1]{Bonahon} showed that the slide diffeomorphisms act transitively on the set of special sphere systems. Therefore, it is enough to show that there is a special sphere system whose $M_0$ contains $S'$. By slightly perturbing $S$, we can assume that the sphere systems $S$ and $S'$ are transverse. From the collection of circles in their intersection, we choose a maximal family of disjoint circles and let $C$ be an {\em innermost} circle in the intersection of $S_i$ and $S'$. Let $D$ be a disk on $S'$ that bounds $C$ and let $D_1$ and $D_2$ be the two disks on $S_i$ that bound $C$. Since $M_i(S)$ is irreducible for $i>0$, if $D$ lies in $M_i(S)$ there will be an isotopy pushing $D$ into $M_0(S)$ to remove the intersection $C$ and possibly others. We eliminate all such intersections. Now suppose $D$ lies in $M_0(S)$.  We use Scharlemann's construction to do surgery on $S_i$ using $D$ to obtain an  embedded sphere $S^*_i$ so that 
\begin{itemize}
\item $(S\backslash S_i) \cup S^*_i$ is a special sphere system.
\item The number of components of $((S\backslash S_i) \cup S^*_i)\cap S'$ is less than the number of components of $S\cap S'$. 
\end{itemize}
We first do surgery on $S_i$ along $D$ to obtain two disjoint spheres $\Sigma_1$ and $\Sigma_2$ that are nearby parallel copies of $D\cup D_1$ and $D\cup D_2$. We shall connect sum these two spheres by a tube around an arc $\alpha$ in $M$ so that $\alpha$ does not intersect $S'$ and $\alpha$ intersects $S$ only at its end points. To choose $\alpha$, note that the components $S'\cap M_i$ are not disks. In a component of $S'\cap M_i$ that is adjacent to $D$, we choose an arc from $\partial D$ to another component $S\cap S'$ and we choose $\alpha$ to be a nearby parallel copy this arc. 

 Let $S^*_i$ be a parametrized embedding of the sphere obtained by connecting sum of $\Sigma_1$ and $\Sigma_2$ along a tube around $\alpha$.   Then one checks that $(S\backslash S_i) \cup S^*_i$ is a sphere system whose intersection with $S'$ has fewer connected components.   See \Cref{pic} for a schematic. By repeating this process, we obtain a special sphere system $S''$ that has no intersection with $S'$.  Therefore, the spheres in $S'$ are either in $M_0(S'')$ or are parallel to sphere boundaries of $M_0(S'')$ which by another isotopy can be moved into $M_0(S'')$. 
\end{proof}

\begin{figure}

\begin{tikzpicture}
\begin{scope}[shift={(-4.1,0)}]
\draw [line width=1.05pt]  (0,0) circle (1cm);
\draw [line width=1.05pt, gray] (1,1) to [out= 195, in=10] (0,0.7);
\draw [line width=1.05pt, gray] (0,0.7) to [out= 190, in=50] (-1.5,1);
\draw [line width=1.05pt, gray] (-1.5,1) to [out= 230, in=110] (-0.4,-0.2);
\draw [line width=1.05pt, gray] (-0.4,-0.2) to [out= -70, in=50] (-0.6,-1.2);
\node at (-1.5, 0) {\small$M_0$};
\node at (0, 0) {\small$M_i$};
\node at (-1, 1.5) {\small$D$};
\node at (1.2, -0.5) {\tiny$S_i$};
\node at (1.2, 1) {\textcolor{gray}{\tiny$S'_i$}};
\end{scope}
\begin{scope}[shift={(0,0)}]
\draw[line width=1.05pt] ($(0,0)+({1*cos(170)},{1*sin(170)})$) arc (170:472:1);

\draw[line width=1.05pt] ($(0,0)+({1*cos(170)},{1*sin(170)})$) to [out= 160, in=230] (-1.6,1.1);
\draw[line width=1.05pt] (-1.6,1.1) to [out= 50, in=160] ($(0,0)+({1*cos(472)},{1*sin(472)})$);

\draw[line width=1.05pt] ($(0,0)+({1*cos(155)},{1*sin(155)})$) to [out= 160, in=230] (-1.4,0.9);
\draw[line width=1.05pt] (-1.4,0.9) to [out= 50, in=160] ($(0,0)+({1*cos(130)},{1*sin(130)})$);

\draw[line width=1.05pt] ($(0,0)+({1*cos(130)},{1*sin(130)})$) arc (130:155:1);
\draw [line width=1.05pt, gray] (1,1) to [out= 195, in=10] (0,0.7);
\draw [line width=1.05pt, gray] (0,0.7) to [out= 190, in=50] (-1.5,1);
\draw [line width=1.05pt, gray] (-1.5,1) to [out= 230, in=110] (-0.4,-0.2);
\draw [line width=1.05pt, gray] (-0.4,-0.2) to [out= -70, in=50] (-0.6,-1.2);
\node at (-1.5, 0) {\small$M_0$};
\node at (0, 0) {\small$M_i$};
\node at (1.2, -0.5) {\tiny$\Sigma_2$};
\node at (-0.5, 0.5) {\tiny$\Sigma_1$};
\node at (1.2, 1) {\textcolor{gray}{\tiny$S'_i$}};
\end{scope}
\begin{scope}[shift={(4.1,0)}]
\draw[line width=1.05pt] ($(0,0)+({1*cos(170)},{1*sin(170)})$) arc (170:260:1);
\draw[line width=1.05pt] ($(0,0)+({1*cos(270)},{1*sin(270)})$) arc (270:472:1);

\draw[line width=1.05pt] ($(0,0)+({1*cos(170)},{1*sin(170)})$) to [out= 160, in=230] (-1.6,1.1);
\draw[line width=1.05pt] (-1.6,1.1) to [out= 50, in=160] ($(0,0)+({1*cos(472)},{1*sin(472)})$);

\draw[line width=1.05pt] ($(0,0)+({1*cos(155)},{1*sin(155)})$) to [out= 160, in=230] (-1.4,0.9);
\draw[line width=1.05pt] (-1.4,0.9) to [out= 50, in=160] ($(0,0)+({1*cos(130)},{1*sin(130)})$);

\draw[line width=1.05pt] ($(0,0)+({1*cos(130)},{1*sin(130)})$) arc (130:138:1);
\draw[line width=1.05pt] ($(0,0)+({1*cos(148)},{1*sin(148)})$) arc (148:155:1);

\draw[line width=1.05pt] ($(0,0)+({1*cos(138)},{1*sin(138)})$) to [out= 320, in=90] ($(0,0)+({1*cos(270)},{1*sin(270)})$);
\draw[line width=1.05pt] ($(0,0)+({1*cos(148)},{1*sin(148)})$) to [out= 330, in=80] ($(0,0)+({1*cos(260)},{1*sin(260)})$);

\draw [line width=1.05pt, gray] (1,1) to [out= 195, in=10] (0,0.7);
\draw [line width=1.05pt, gray] (0,0.7) to [out= 190, in=50] (-1.5,1);
\draw [line width=1.05pt, gray] (-1.5,1) to [out= 230, in=110] (-0.4,-0.2);
\draw [line width=1.05pt, gray] (-0.4,-0.2) to [out= -70, in=50] (-0.6,-1.2);
\node at (-1.5, 0) {\small$M_0$};
\node at (0.5, 0) {\small$M_i$};
\node at (1.2, -0.5) {\tiny$S^*_i$};
\node at (1.2, 1) {\textcolor{gray}{\tiny$S'_i$}};
\end{scope}
\end{tikzpicture}
\caption{Scharlemann's surgery on sphere systems}\label{pic}
\end{figure}
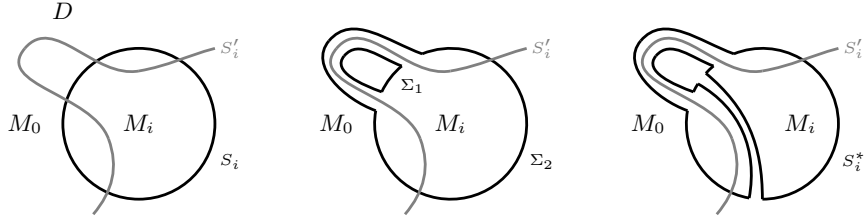
Now we are ready to prove that the $0$-th row of the $E_1$-page of the spectral sequence \ref{sss} is acyclic.

\begin{proof}[Proof of \Cref{0th row}]
Recall from \ref{ssss} that 
\[
E^{*,0}_1(X_{\bullet}(M)\hcoker \SDiff(M))=H^0(X_{\bullet}(M)\hcoker \SDiff(M);\bQ).
\]
Since we work with rational coefficients, it is enough to prove the dual statement that the chain complex
\begin{equation}\label{s''}
(H_0(X_{\bullet}(M)\hcoker \SDiff(M);\bQ), d_1),
\end{equation}
is acyclic. But note that the set of connected components  $\pi_0(X_{\bullet}(M)\hcoker \SDiff(M))$ is isomorphic to the set of orbits of the action of $\pi_0(\SDiff(M))$ on $\pi_0(X_{\bullet}(M))$.  Let us denote the semisimplicial set $\pi_0(X_{\bullet}(M)\hcoker \SDiff(M))$ by $K_{\bullet}(M)$. Then we have 
\[
(H_0(X_{\bullet}(M)\hcoker \SDiff(M);\bQ), d_1)\cong (\bQ[K_{\bullet}(M)], d_1),
\]
where $d_1$ is induced by  the alternate sum of face maps of $K_{\bullet}(M)$. But this chain complex calculates the homology of the realization $|K_{\bullet}(M)|$. Hence, it is enough to show that the realization of the semisimplicial set $K_{\bullet}(M)$ is contractible. Since $K_{\bullet}(M)$ is a set, the realization $|K_{\bullet}(M)|$ has a $\Delta$-complex structure (see \cite[Remark 6.3]{galatius2012stable}). So any map $f: S^n\to |K_{\bullet}(M)|$ can be homotoped to be simplicial for a triangulation of $S^n$. Hence, this map hits finitely many vertices $v_1,v_2,\cdots, v_k$ in $K_0(M)$. So if we show that for any such finitely many vertices, there exists a vertex $v$ in $K_0(M)$  that is adjacent to all $v_i$ then we can extend  $f$ to the join
\[
f*\{ v\}: S^n*\{v\}\to |K_{\bullet}(M)|,
\]
which implies that $f$ is nullhomotopic. Since the union of spheres in the sphere systems of a $p$-simplex for all $p$ constitute a sphere system again, by \Cref{HM}, each orbit has a representative of parametrized spheres in $M_0(S)$. We choose representatives of sphere systems $v_i$ inside $M_0(S)$. Note that these representatives are disjoint from the sphere system $S$. Let $v$ be the sphere system by adding parallel spheres to the sphere system $S$ so that it lies in a different orbit than that of $v_i$'s. We can choose the parallel spheres so that the spheres in $v$ are still disjoint from the representative sphere systems of $v_i$'s. Hence, $v$ is adjacent to $v_i$'s in $  |K_{\bullet}(M)|$. Therefore, $f$ is nullhomotopic. 
\end{proof}
\section{Dynamical obstructions to extending diffeomorphisms}  \label{sec:dynamical}

This section gives an alternative approach to extension problems, using the dynamics of group actions (specifically, fixed sets of finite order elements) to obstruct extensions, with arguments in the style of \cite{Ghys}.  We treat the torus boundary case, followed by the proof for sphere boundary.     

Proposition \ref{prop:irreducible} shows that when $M$ is irreducible and not equal to $D^2 \times S^1$, there is a finitely generated subgroup of $\Diff_0(S^1)\times \{ \mathrm{id}\} \subset \Diff_0(S^1 \times S^1)$ that does not lift to $\Diff_0^1(M)$.  In the next proposition, we show that this is true for all other manifolds with $\partial M \cong T^2$, using a dynamical rather than cohomological approach.  
This gives the following.

\begin{prop} \label{prop:solid_torus}
Let $M$ be a 3-manifold with $\partial M \cong T^2$. 
There is a finitely generated subgroup  $\Gamma \subset \Diff_0(\partial M)$ that does not lift to $\Diff_0^1(M)$.  In fact, we may find such a finitely generated group contained in the subgroup $\Diff_0(S^1)\times \{ \mathrm{id}\} \subset \Diff_0(S^1 \times S^1)$.
\end{prop}

\begin{proof} 
As remarked above, we need only treat the case where $M$ is reducible or where $M = D^2 \times S^1$.
Assume first that $M$ is reducible. Following \cite{Ghys}, we may find elements $f$ and $g$ in $\Diff_0(S^1)$ satisfying the following relations:
\[ 
[f,g]^{6} = \mathrm{id}, \,  [[f,g]^2, f] = [[f,g]^2, g] = \mathrm{id},
\]
by writing an order $2$ rotation of $S^1$ as a commutator of two elements $\bar{f}$ and $\bar{g}$ in $\mathrm{PSL}(2,\mathbb{R}) \subset \Diff_0(S^1))$, and then choosing $f$ and $g$ to be any lifts of  $\bar{f}$ and $\bar{g}$, respectively,  to diffeomorphisms of a 3-fold cover of the circle. Note that this ensures that the commutator $[f,g]$ has order 6.

Identify $f$ and $g$ with diffeomorphisms of $\Diff_0(S^1\times S^1)$ acting trivially on the second $S^1$ factor.   Let $G$ denote the group generated by  $f$ and $g$.  We will now show that $G$ admits no extension to $\Diff_0(M)$.   
Suppose for contradiction that $\phi\colon G \to \Diff_0(M)$ were an extension.  Let $r$ denote the commutator $[f,g]$, so $\phi(r)$ is an order $6$ diffeomorphism of $M$.  We show first that the set of points fixed by $\phi(r)^2$ is nonempty.  Note that $\phi(r)^2$ is finite order and orientation preserving, so its fixed set is either $0$ or $1$-dimensional. 

By the equivariant sphere theorem \cite{Dunwoody}, there exists a reducing system of spheres that is setwise preserved by the finite order diffeomorphism $\phi(r)$, with $\phi(r)$ permuting the spheres in the system.  
Since $\phi(\Gamma)$ preserves the boundary torus, it preserves the sphere bounding the irreducible component with boundary torus.  Since $\phi(r)^2$ has order 3, its action on this invariant sphere is conjugate to a rotation (this is true even for actions on spheres by homeomorphisms, due to a result of Kerekjarto \cite{Constantin_Kolev}) and so it fixes exactly two points on this sphere.  Since $\phi(r)^2$ also preserves the tangent plane to these two points, we conclude that $\fix(\phi(r)^2)$ is 1-dimensional, hence a union of finitely many disjoint circles in $M$.  Finally, since $\phi(f)$ and $\phi(g)$ commute with $\phi(r)^2$, they preserve its fixed set. 

Choose local coordinates on $M$ that identify a tubular neighborhood of $\fix(\phi(r)^2)$ with a disjoint union of copies of $D^2 \times S^1$ on which $\phi(r)^2$ acts by an order 3 rotation of each disk $D^2 \times \{x \}$ about $0$.  
In particular, in these coordinates the derivative of $\phi(r)^2$ at each fixed point is the linear map represented by the block matrix
$\left( \begin{smallmatrix} A & 0 \\ 0 & 1\end{smallmatrix} \right)$ where $A$ is a nontrivial order 3 element of $\mathrm{SO}(2)$.    
Since $f = r^2 fr^{-2}$, the derivative of $\phi(f)$ at a point in $\mathrm{Fix}(\phi(r)^2)$ commutes with $D\phi(r)^2 = \left( \begin{smallmatrix} A & 0 \\ 0 & 1\end{smallmatrix} \right)$.  But the centralizer of this matrix consists of matrices of the form $\left( \begin{smallmatrix} B & 0 \\ 0 & t\end{smallmatrix} \right)$, where $B \in \mathrm{O}(2)$ and $t \in \bR$; an abelian subgroup of $\mathrm{GL}(3,\bR)$.  The same is true for $g$, so we have $D\phi([f,g]) = I$ at any point $x \in \fix(\phi(r)^2)$ contradicting the fact that $D\phi([f,g])^2= \left( \begin{smallmatrix} A & 0 \\ 0 & 1\end{smallmatrix} \right)$.

The proof is similar in the case where $M= D^2 \times S^1$.  
Take $f$ and $g$ exactly as above, let $r = [f,g]$ and suppose again for contradiction that $\phi$ were an extension of the action to $\Diff_0(D^2 \times S^1)$.  Then $\phi(r)$ is an order 6 diffeomorphism of $D^2\times S^1$ preserving (setwise) each circle of the form $S^1 \times \{x\}$.  
We claim that $\phi(r)$ has nonempty fixed set, with $\fix(\phi(r))$ a topological circle.  
One way to see this is to lift the action of $\phi(r)$ to an order 6 diffeomorphism of the universal cover $D^2 \times \mathbb{R}$ rotating each circle $\partial D^2 \times \{x\}$, which we may extend to a diffeomorphism of $\mathbb{R}^3$ acting as a rotation about the $z$-axis outside of $D^2 \times \mathbb{R}$.   Averaging a metric so that $\phi(r)$ and its iterates act by isometries, it must preserve and act as an order 6 rotation on each sphere about $0$, hence has two fixed points on the sphere. The union of these fixed points forms the axis of $\phi(r)$.   
 
We can then follow the argument from the previous case above verbatim, trivializing the unit tangent bundle in a neighborhood of $\fix(\phi(r))$, and thus derive a contradiction.   
\end{proof}

We now prove the second item from Theorem \ref{thm:torus_general}. 
Recall this was the statement that there is a finitely generated group which acts on $T^2$ by isotopically trivial diffeomorphisms such that the action admits no extension to $\Diff_0(M)$ for any $3$-manifold $M$ bounded by $T^2$.  
\begin{proof}
Let $\Gamma$ be the group defined above, and let $\Gamma' = \Gamma \times \Gamma$ be the direct product of $\Gamma$ with itself.  Recall that $\Gamma$ acted on $S^1 \times S^1$ with a trivial action on one of the factors  (that is to say, it was naturally a group of homeomorphisms of $S^1$) so $\Gamma'$ has an obvious product action on $S^1 \times S^1$.  
We will show that this action of $\Gamma'$ does not extend to any 3-manifold $M$ with $\partial M \cong T^2$.  Proposition \ref{prop:solid_torus} shows this when $M$ is reducible or the solid torus; in that case the subgroup $\Gamma \times \{1\}$ does not even extend.  

In the case where $M$ is irreducible, we will appeal to Theorem \ref{main}. Using this, it suffices to show that the pullbacks of  the two Euler classes in $\Diff_0(S^1 \times S^1)$ to $\Gamma'$ are linearly independent in $H^*(\Gamma';\bQ)$.  
To see this, let $\Gamma_{6} :=\langle a_1, b_1, \ldots a_{6}, b_{6} \mid \prod_i [a_i, b_i] \rangle$ be the fundamental group of a genus 6 surface with its standard presentation.  There is a homomorphism $\rho$ from $\Gamma_{6}$ to the group $\Gamma \subset \Diff_0(S^1)$ by sending $a_i$ to the homomorphism $f$, for each $i$, and 
 $b_i$ to $g$.   If $f$ and $g$ are chosen so that $[f,g]$ is a standard rotation by $2\pi/3$, then it is easily verified (for example, this follows from the computation of \cite{Milnor}) that the pullback of the Euler class under $\rho$ pairs with the fundamental class in $H_2(\Sigma_6; \mathbb{Z}) \cong H_2(\Gamma_6; \mathbb{Z})$ to give 1, so in particular the Euler class from $\Diff_0(S^1)$ pulls back nontrivially to $\Gamma$.   Since $\Gamma'$ is a product action of $\Gamma$ on each factor, one may embedd $\Gamma_6$ into $\Gamma \times \{1\} \subset \Gamma'$ as above, or into $\{1\} \times \Gamma$.  Considering the pullback of the Euler classes under each embedding shows that the Euler classes pull back to linearly independent elements in $H^*(\Gamma \times \Gamma;\bQ)$, which was what we needed. 
\end{proof}   

We now treat the case of manifolds with sphere boundary.  

\begin{prop}
Let $M$ be a 3-manifold with $\partial M \cong S^2$.  Then there is no extension $\Diff_0(S^2) \to \Diff^1(M^3)$.  
\end{prop}

The proof here is inspired by and adapted from Ghys' proof for $M \cong B^3$.  

\begin{proof}  
For concreteness, parametrize $S^2$ as the unit sphere in $\mathbb{R}^3$.  
Identify $\SO(2)$ with the subgroup of $\Diff(S^2)$ consisting of rotations about the $z$-axis.  Let $n, s$ be the fixed points of these rotations.  
For $r \in \SO(2)$, denote by $G_r$ the centralizer of $r$ in $\Diff_c(S^2 - \{n, s\})$. 

Let $f$ and $g$ in $\SO(2)$ be the rotations of order 2 and 3 respectively. The first tool is a lemma proved by Ghys.  

\begin{lem}[\cite{Ghys}, Lemma 4.4]\label{lem:generate}
$\Diff_c(S^2 - \{n, s\})$ is generated by $G_f \cup G_g$. 
\end{lem}
Now suppose that $M$ is a 3-manifold with $\partial M = S^2$.  Suppose that we have an extension $\phi\colon \Diff_0(S^2) \to \Diff_0(M^3)$.  We will ultimately derive a contradiction by finding a finite order element $h \in \Diff(S^2)$ such that $\phi(h)$ has a fixed point at which its derivative is the identity, contradicting that $\phi(h)$ must be nontrivial and finite order. 

First we study the fixed set of $\phi(f)$.  This is a one dimensional manifold with boundary embedded as a submanifold of $M$.  
As $G_f$ commutes with $f$, $\phi(G_f)$ preserves $\fix(\phi(f))$, so there is a homomorphism $G_f \to \Homeo(\fix(\phi(f)))$.  Since $G_f$ is isomorphic to the group of compactly supported homeomorphisms of an open annulus, by \cite{Mann15}, this homomorphism must be trivial. Moreover, at each point $x \in \fix(\phi(f))$ we have a homomorphism of $G_f$ to $\GL(3,\mathbb{R})$ by taking derivatives.  Since $G_f$ is a simple group (this is a deep result following from \cite{thurston1974foliations}), this homomorphism is trivial.  
The same reasoning applies to show that $\phi(G_g)$ acts trivially on $\fix(\phi(g))$, with trivial derivatives.  

Since $\fix(\phi(f))$ is an embedded 1-manifold with boundary in $M$, and $\fix(\phi(f)) \cap \partial M = \{n, s\}$, there is a unique connected component of $\fix(\phi(f))$ that is diffeomorphic to a closed interval.  Let $I$ denote this interval; its endpoints are $n$ and $s$.
The same reasoning applies to $\phi(g)$, and since $g$ and $f$ commute, $\phi(g)$ preserves $I$ so $I$ must be equal to the interval component of $\fix(\phi(g))$ as well.  Thus, our reasoning above, combined with Lemma \ref{lem:generate} implies that for every point $x \in I$, $\phi(\Diff_c(S^2 - \{n, s\})$ fixes $x$ and has trivial derivatives. 

Let $h$ be an order 2 diffeomorphism that is a rigid rotation commuting with $f$ but rotating about the orthogonal $y$-axis.  Let $e$ and $w$ be the fixed points of $h$.  Since $\phi(h)$ preserves $\fix(\phi(f))$ and exchanges $n$ and $s$, it follows that $\phi(h)$ acts on $I$ as an orientation reversing diffeomorphism, with a unique fixed point.  Let $x_0$ denote this fixed point.   Extending the use of our previous notation, let $G_h$  denote the centralizer of $h$ in $\Diff_c(S^2 - \{e, w\})$.  Then by our argument above, $G_h$ fixes $x_0$ and has trivial derivatives there.   

Finally, let $s_1 \in G_h$ agree with $h$ on the annulus $\{(x,y,z) \in S^2 \mid y \in [-1/2, 1/2]\}$ and act as a rotation on each circle $y = c$, smoothly interpolating between the order 2 rotation on $y \in [-1/2, 1/2]$ and the identity on neighborhoods of $y=-1$ and $y = 1$.   Then $s_1^{-1} h \in \Diff_c(S^2 - \{n, s\})$.  Thus, $\phi(h) = \phi(s_1) \circ \phi(s_1^{-1} h)$ fixes $x_0$ with trivial derivative, giving the desired contradiction.  
\end{proof}

This proof, much like Ghys' proof for $M = D^3$, uses simplicity of the group of compactly supported diffeomorphisms of an open disk.  This is not known for many other natural groups of diffeomorphisms, for instance the group of real analytic diffeomorphisms.  It would be interesting to know whether this group similarly fails to admit an extension.   Restricting to smaller (e.g. finitely generated) subgroups, as in the following problem, makes the section problem even more challenging.   
\begin{problem} Find a finitely generated group $\Gamma \subset \Diff_0(S^2)$ with no extension to $D^3$.  More generally, does there exist a finitely generated group $\Gamma \subset \Diff_0(S^2)$ such that its action on $S^2$ is not nullbordant, in the sense of Definition \ref{bordism def}? 
\end{problem}


\section{Cohomological obstructions for manifolds with sphere boundary}  \label{sec:p_1}

In this section we address the following case of our general program to find cohomological obstructions to extension.  
\begin{problem}\label{quest:p_1}
For which $3$-manifolds $M$ where $\partial M \cong S^2$,
 is the image of the first Pontryagin class $p_1$ under the map $H^4(\mathrm{BSO}(3);\bQ)\to H^4(\BDiff_0(M);\bQ)$ zero?
\end{problem}

Suppose that $M$ is a $3$-manifold whose boundary $\partial M$ is diffeomorphic to $S^2$.  Smale's theorem \cite{MR0112149}, states that $\Diff_0(S^2)\simeq \mathrm{SO}(3)$.  Thus, if we require extensions to be {\em continuous}, we can prove no such extension exists simply by showing that  the map
\[
\text{res}^*\colon H^*(\mathrm{BSO}(3);\bQ)\to H^*(\BDiff_0(M);\bQ),
\]
that is induced by restriction to the boundary, has a nontrivial kernel. 

Requiring continuity was not in the original spirit of Ghys' question -- he asks this from a purely algebraic perspective -- but 
recent automatic continuity results imply that, in many cases, all extensions of the action of the full group $\Diff_0(\partial M)$ are necessarily continuous.   These were used in \cite{ChenMann} to give a negative general answer to the original question as phrased in the introduction, using a completely different approach to that here.   Given this, it would be very interesting to know in which cases such obstructions to extension are also cohomological in nature, and whether there are smaller topological subgroups (i.e. proper subgroups of $\Diff_0(\partial M)$) which fail to extend.  For a prime $3$-manifold $M$, the homotopy type of $\Diff_0(M)$ as a topological group is very well studied. We use this knowledge to find cases that exhibit cohomological obstructions for continuous group extenstions and hence by the following automatic continuity results for all group extensions.

For smooth diffeomorphisms, continuity follows from a result of Hurtado. 
\begin{thm}[Hurtado \cite{Hurtado}]
Let $M$ and $N$ be closed smooth manifolds.  Then any homomorphism $\Diff_0(M) \to \Diff_0(N)$ is continuous. 
\end{thm}
To apply this in our situation, let $N$ be the double of $M$ along the boundary, and note that any extension $\Diff_0(\partial M) \to \Diff(M)$ induces a homomorphism $\Diff_0(\partial M) \to \Homeo_0(N)$ by doubling.  However, a smoothing trick (see \cite{Parkhe}) permits one to conjugate the extension of the action in such a way that the gluing becomes smooth at the boundary, producing a homomorphism $\Diff_0(\partial M) \to \Diff_0(N)$; which by Hurtado's theorem must be continuous.  It follows that the extension must be continuous.  
The situation is similar for homeomorphisms (and one does not even need to make a gluing argument) due to work of the first author.  
\begin{thm}[\cite{Mann_automatic}]
Let $M$ be a compact manifold, and $G$ any separable topological group. Then any homomorphism $\Homeo_0(M) \to G$ is continuous.  
\end{thm}
\noindent Since homeomorphism groups of compact manifolds are separable, this shows any extension is necessarily continuous.  Interestingly, the case of continuity between maps of $C^r$ diffeomorphisms of manifolds, for $0 < r < \infty$, remains open.

We now give some sample results where Question \ref{quest:p_1} can be answered with existing machinery.  
For this, it is easier to work with a marked point instead of the sphere boundary. To change the map 
\[
\text{res}^*\colon H^*(\mathrm{BSO}(3);\bQ)\to H^*(\BDiff_0(M);\bQ),
\]
to the derivative map at a marked point, we first recall the following low dimensional fact.
\begin{lem}\label{lem:2}
For a closed $3$-manifold $P$, the group $\Diff_0(P\backslash D^3)$ has the same homotopy type as $\Diff_0(P\backslash\text{int}(D^3))$.
\end{lem}
\begin{proof}
Consider the zig-zag of maps
\[
\Diff_0(P\backslash D^3)\xrightarrow{\simeq} \mathrm{Homeo}_0(P\backslash D^3)\xleftarrow{} \mathrm{Homeo}_0(P\backslash\text{int}(D^3))\xleftarrow{\simeq}\Diff_0(P\backslash\text{int}(D^3)).
\]
Let $x$ be the center of the embedded ball $D^3$ in $P$. The group $ \mathrm{Homeo}_0(P\backslash D^3)$ has the same homotopy type as $ \mathrm{Homeo}_0(P\backslash x)\cong  \mathrm{Homeo}_0(P, \text{rel }x)$. 

On the other hand, from the case $2$ in the proof of \Cref{prop:20} and Cerf's theorem, we know 
\[
\mathrm{Homeo}_0(P\backslash\text{int}(D^3))\xleftarrow{\simeq}\Diff_0(P\backslash\text{int}(D^3))\xleftarrow{\simeq}\Diff_0(P\backslash\text{int}(D^3), \partial_{\text{SO}(3)}),
\]
\[
\Diff_0(P\backslash\text{int}(D^3), \partial_{\text{SO}(3)})\xrightarrow{\simeq}\Diff_0(P,\text{rel }x)\xrightarrow{\simeq} \mathrm{Homeo}_0(P, \text{rel }x).
\]
Therefore, there is a zig-zag of weak homotopy equivalences between $\Diff_0(P\backslash D^3)$ and  $\Diff_0(P\backslash\text{int}(D^3))$, as desired. 
\end{proof}

As observed in the proof of \Cref{lem:2}, we have $\Diff_0(M)\simeq \Diff_0(N,\text{rel }x)$ where $N$ is a closed $3$-manifold obtained from $M$ by capping of the boundary sphere with a ball whose center is $x$.
Hence, to show that the action of $\Diff_0(\partial M)$ does not extend to $\mathrm{Homeo}_0(M)$, it is enough to show that the map 
\[
H^*(\mathrm{BSO}(3);\bQ)\to H^*(\BDiff_0(N, \text{rel }x);\bQ),
\]
that is induced by taking derivative at $x$, has a non-trivial kernel.  

\paragraph{Sample application} 
As a toy case (given known deep results), and to give an example of this approach, we give an alternative proof of Theorem \ref{thm:sphere} when $M$ is obtained from a hyperbolic $3$-manifold or a Haken manifold by removing a ball.
\begin{prop} \label{prop:2}
Let $N$ be a closed, irreducible, hyperbolic or Haken $3$-manifold. Let $x\in N$ be a marked point.  Then the image of the first Pontryagin class $p_1$ under the map $$H^4(\mathrm{BSO}(3);\bQ)\to H^4(\BDiff_0(N, \text{rel }x);\bQ),$$ induced by taking derivative at $x$, is zero. 
\end{prop}

\begin{cor} 
For $M$ that is obtained by removing a ball from $N$ as above, there is no extension $\Diff_0(\partial M) \to \mathrm{Homeo}_0(M)$.  
\end{cor} 

\begin{rem}
The corollary follows easily in the case where $N$ is hyperbolic since there is a bound on the order of a finite order diffeomorphism of a hyperbolic $N$-manifold, hence there are finite subgroups of $\mathrm{SO}(3)$ that will not extend.  
In detail, if $f$ is a finite order element of $\mathrm{SO}(3)$ that extends to a diffeomorphism of $M$, we may extend this to a finite order diffeomorphism of $N$ acting as a rotation on the ball.  Thus, its fixed set is 1-dimensional, in which case work of Thurston shows that it is conjugate to an isometry of $N$ with a hyperbolic metric.  Mostow rigidity now gives a bound on the order of $f$. 

The new content in this case of Proposition \ref{prop:2} is the cohomological obstruction to extension.
\end{rem}

\begin{proof}[Proof of \Cref{prop:2}]If $N$ is hyperbolic, by Gabai's theorem (\cite{gabai2001smale}), we have $\Diff_0(N)\simeq *$  and if $N$ is Haken, by Hatcher's theorem (\cite{hatcher1999spaces}), if $N$ is  diffeomorphic to a $3$-torus then the natural inclusion $N\hookrightarrow \Diff_0(N)$ is a homotopy equivalence, and otherwise we have $\Diff_0(N)\simeq *\text{ or } S^1$.
 
 \noindent {\bf Case 1}: Suppose $\Diff_0(N)\simeq *$. We have a fibration 
 \begin{equation}\label{fib}
 N\to \BDiff_0(N, \text{ rel }x)\to \BDiff_0(N).
 \end{equation}
Therefore, $\BDiff_0(N, \text{ rel }x)$ has the same homotopy type as $N$. Hence, we have $H^4(\BDiff_0(N, \text{ rel }x);\bQ)=0$, in particular, the image of $p_1$ under the derivative map vanishes.

 \noindent {\bf Case 2}: Suppose $\Diff_0(N)\simeq S^1$. Hence, the fibration \ref{fib}, is the same as the following fibration up to homotopy
 \[
 N\to N\hcoker S^1\to \mathrm{B}S^1.
 \]
Because $\Diff_0(N)\simeq S^1$, the manifold $N$ is a Seifert fibered manifold with a free $S^1$ action. Therefore, the homotopy quotient $N\hcoker S^1$ is homotopy equivalent with the quotient $N/S^1$ which is a $2$-dimensional CW-complex. Hence, again we have 
\[
H^4( \BDiff_0(N, \text{ rel }x);\bQ)=H^4(N\hcoker S^1;\bQ)=0.
\]
So the image of $p_1$ under the derivative map vanishes.

 \noindent {\bf Case 3}: Suppose $N$ is diffeomorphic to a $3$-torus. Since $\Diff_0(N)$ is homotopy equivalent to $N$, the fibration \ref{fib} implies that $\BDiff_0(N, \text{ rel }x)$ is contractible. Hence, the image of $p_1$ under the derivative map vanishes.
\end{proof}

\begin{rem} 
We remark that is not hard to show that, when $M$ is a lens space with a $3$-ball removed, the first Pontryagin class does not vanish.
It appears to be an interesting problem to answer problem \ref{quest:p_1} for {\em reducible} manifolds.  
\end{rem} 

\bibliographystyle{alpha}
\bibliography{reference}
\end{document}